\documentclass{article}
\usepackage{amsmath,amssymb}
\usepackage{amsthm}
\usepackage{graphicx}
\usepackage{subfigure}

\usepackage[utf8]{inputenc}
\usepackage[T1]{fontenc}
\usepackage{lmodern}

\usepackage{tikz}
\usepackage[arrow, matrix, curve]{xy}

\usepackage{hyperref}

\usepackage{slashed} 

\numberwithin{equation}{section}

\newtheorem{theorem}{Theorem}[section]

\newtheorem{lemma}[theorem]{Lemma}
\newtheorem{proposition}[theorem]{Proposition}
\newtheorem*{theorem*}{Theorem}
\newtheorem*{main*}{Main Theorem}

\theoremstyle{definition}
\newtheorem{definition}[theorem]{Definition}
\newtheorem{remark}[theorem]{Remark}

\begin{document}
\title{The Banach manifold $C^k(M,N)$}
\author{Johannes Wittmann}

\maketitle
\begin{abstract}
	Let $M$ be a compact manifold without boundary and let $N$ be a connected manifold without boundary. For each $k\in\mathbb{N}$ the set of $k$ times continuously differentiable maps between $M$ and $N$ has the structure of a smooth Banach manifold where the underlying manifold topology is the compact-open $C^k$ topology. We provide a detailed and rigorous proof for this important statement which is already partially covered by existing literature.
\end{abstract}

\tableofcontents\newpage

\section{Introduction}
Let $M$ be a closed manifold\footnote{By ``manifold'' we always mean a finite-dimensional manifold with or without boundary. All manifolds we consider are non-empty, second-countable, and Hausdorff. All manifolds considered are smooth (= $C^\infty$), unless otherwise specified. A closed manifold is a compact manifold without boundary. Moreover, in the following we use ``vector space'' for vector spaces over $\mathbb{R}$.} and let $N$ be a connected manifold without boundary. For each $k\in\mathbb{N}:=\{0,1,2,\ldots\}$ we denote by $C^k(M,N)$ the set of $k$ times continuously differentiable maps between $M$ and $N$.

It is well known that for each $k\in\mathbb{N}$ the set $C^k(M,N)$ has the structure of a smooth Banach manifold. The natural idea to turn $C^k(M,N)$ into a Banach manifold is to choose a Riemannian metric on $N$ and then use the exponential map of $N$ to construct the charts of $C^k(M,N)$. More precisely, for $g$ close enough to $f$, the map
\[C^k(M,N)\ni g\mapsto (p\mapsto (\textup{exp}_{f(p)})^{-1}g(p))\in \Gamma_{C^k}(f^*TN),\]
is a chart around $f$. Here, $\textup{exp}$ denotes the exponential map of the Riemannian manifold $N$. This idea can be found in many places in the literature (references are given below). Let us denote this chart by $\varphi_f$.

Driven by applications, there are several natural requirements and questions: One needs a rigorous and detailed proof that these charts induce a smooth structure. Are the transition maps $\varphi_f\circ (\varphi_g)^{-1}$ only smooth for $f,g\in C^\infty(M,N)$ or are they also smooth in the case that $f$ and $g$ are precisely $k$ times continuously differentiable? Is the manifold topology of $C^k(M,N)$ the compact-open $C^k$ topology?

An investigation of literature regarding such questions only brought up partial answers and proofs \cite{Eells1, Blue, Eells, Eliasson, Palais, Michor, Hamilton, MTAA, KM}. We explain this in more detail at the end of this section. Note that the case $k=\infty$ is better dealt with in the literature, in particular a very thorough treatment of the space $C^\infty(M,N)$ can be found in \cite{KM}.

In this paper we provide a detailed proof for the following theorem.

\begin{main*}[c.f. Theorem \ref{theorem ckmn banach mf}]Let $M$ be a closed manifold and let $N$ be a connected manifold without boundary. Let $k\in \mathbb{N}$ and fix a Riemannian metric on $N$. Then the set $C^k(M,N)$ endowed with the compact-open $C^k$ topology has the structure of a smooth Banach manifold with the following property: for any $f\in C^k(M,N)$ and any small enough open neighborhood $U_f$ of $f$ in $C^k(M,N)$ there is an open neighborhood $V_f$ of the zero section in $\Gamma_{C^k}(f^*TN)$ such that the map
	\begin{align*}
	\varphi_f\colon U_f&\to V_f,\\
	g&\mapsto \textup{exp}^{-1}\circ (f,g),
	\end{align*}
i.e., $\varphi_f(g)(p)=(\textup{exp}_{f(p)})^{-1}g(p)$ for all $g\in U_f$, $p\in M$, is a local chart. Here, we endow the space $\Gamma_{C^k}(f^*TN)$ of $C^k$-sections of $f^*TN$ with the usual $C^k$-norm. Note that the inverse of $\varphi_f$ is given by $$\varphi_f^{-1}(s)(p)=\textup{exp}_{f(p)}s(p)$$ for all $s\in V_f$, $p\in M$. Moreover, this smooth structure on $C^k(M,N)$ does not depend on the choice of Riemannian metric on $N$.
\end{main*}


Our detailed treatment of the proof of the Main Theorem might also be helpful for treating mapping groupoids of $C^k$-maps \cite{gr1, gr2}.

The basic strategy to prove the Main Theorem is as follows. We first show that the maps $\varphi_f\colon U_f\to V_f$ are homeomorphisms. Then we argue why the transition maps $\varphi_f\circ \varphi_g^{-1}$ given by
\[(\varphi_f\circ \varphi_g^{-1})(s)=\left((\textup{exp}_{f})^{-1}\circ\textup{exp}_{g}\right)s \]
are smooth provided that $U_f\cap U_g\neq\varnothing$. For this our arguments are inspired by \cite{Blue}. The smoothness of the transition maps is the most delicate part, and one has to argue very carefully, since $\varphi_f$ and $\varphi_g$ are defined using not necessarily smooth maps $f$ and $g$. The main input for this will be the $\Omega$-lemma (using the terminology of \cite{Blue,MTAA}) which we will first prove in a ``local'' version, see Lemma \ref{lemma loc omega}, and then ``globalize'' to maps between sections of vector bundles, see Lemma \ref{lemma glob omega}.

In \cite{Eells} the idea how the charts of $C^k(M,N)$ are constructed is outlined, it is however not included how to show that the charts are homeomorphisms or how to show smoothness of the transition maps. In \cite{Eells1} one finds details for the case $k=0$, i.e., $C^0(M,N)$, but not for general $k\in\mathbb{N}$. The notes \cite{Blue} contain details regarding the proof of the smoothness of the transition maps, however, the question whether the topology of $C^k(M,N)$ is the compact-open $C^k$ topology is not treated.

\section{Preliminaries and the local $\Omega$-lemma}
We begin by recalling some basic definitions regarding the notion of differentiability of maps between normed vector spaces that we use in the following, see e.g. \cite{MTAA}. Let $(X,\|.\|_X)$ and $(Y,\|.\|_Y)$ be normed vector spaces, $U\subset X$ open, and $f\colon U\rightarrow Y$ a map. We say that $f$ is \textit{differentiable at} $x_0\in U$ if there exists a continuous linear map $Df(x_0):=Df_{x_0}\colon X\rightarrow Y$ s.t. for every $\varepsilon >0$, there exists $\delta=\delta(\varepsilon)>0$ s.t. whenever $0<\|x-x_0\|_X<\delta$, we have
\[\frac{\|f(x)-f(x_0)-Df_{x_0}(x-x_0)\|_Y}{\|x-x_0\|_X}<\varepsilon.\]
Moreover, the map $f$ is \textit{differentiable} if $f$ is differentiable at every $x_0\in U$. We say that $f$ is \textit{continuously differentiable} if $f$ is differentiable and the map
\[Df\colon U\rightarrow L(X,Y),\qquad x\mapsto Df_x,\]
is continuous. Here, $L(X,Y)$ denotes the space of continuous linear maps $X\to Y$. Similarly, $L^k(X,Y)$ denotes the space of $k$-multilinear continuous maps $\underbrace{X\times\ldots\times X}_{k\text{ times}}\to Y$. We endow $L^k(X,Y)$ with the norm
\[\|f\|:=\sup\left\{\frac{ \|f(x_1,\ldots,x_k)\|_Y}{\|x_1\|_X\cdot\ldots\cdot \|x_k\|_X}\mid x_1,\ldots,x_k\in X\setminus\{0\}\right\}.\]
Then $L^k(X,Y)$ is a Banach space if $Y$ is a Banach space. Finally, we denote by $L^k_s(X,Y)\subset L^k(X,Y)$ the symmetric elements of $L^k(X,Y)$. Inductively, we define
\[D^kf:=D(D^{k-1}f)\colon U\rightarrow L^k(X,Y)\]
if it exists, where we have identified $L(X,L^{k-1}(X,Y))$ with $L^k(X,Y)$ via the norm-preserving isomorphism
\[L(X,L^{k-1}(X,Y))\ni f\mapsto \bigg((x_1,\ldots,x_k)\to f(x_1)(x_2,\ldots,x_k)\bigg)\in L^k(X,Y) .\] If $D^kf$ exists and is continuous, we say that $f$ is \textit{$k$ times continuously differentiable} (or \textit{$f$ is a $C^k$-map}). We use the notation
$$C^k(U,Y):=\{f\colon U\to Y\text{ }| \text{ } f \text{ is } k \text{ times continuously differentiable}\}.$$
Note that if $f\in C^k(U,Y)$, then $D^kf(x)\in L^k_s(X,Y)$ for all $x\in U$.

In the following the special case $X=\mathbb{R}^n$ will also be important. Then a map $f\colon U \to Y$ (where $U\subset\mathbb{R}^n$ is open and $(Y,\|.\|_Y)$ be a normed vector space) is continuously differentiable iff for all $j=1,\ldots,n$ and all $x_0\in U$ the limit
\[\left(\partial_{x_j}f\right)(x_0):=\lim_{h\to 0}\frac{1}{h}\left(f(x_0+he_j)-f(x_0)\right)\]
exists in $Y$ and the maps $\partial_{x_j}f\colon U\rightarrow Y$ are continuous. Let $k\in \mathbb{N}_{>0}$. Then $f$ is $k$ times continuously differentiable iff for all $j=1,\ldots,n$ the map $\partial_{x_j}f\colon U\rightarrow Y$ is continuous for $k=1$, respectively $(k-1)$ times continuously differentiable for $k\ge 2$. We define 
\begin{align*}
C^k(\overline{U},Y):=\{ f\in C^k(U,Y)\text{ }|\text{ } \partial^\alpha_xf \text{ has a continuous extension to } \overline{U} \text{ for all }|\alpha|\le k\},
\end{align*}
where $\alpha=(\alpha_1,\ldots,\alpha_n)\in\mathbb{N}^n$ is a multiindex, $\partial^\alpha_x=\partial^{\alpha_1}_{x_1}\ldots \partial^{\alpha_n}_{x_n}$, and $|\alpha|=\alpha_1+\ldots+\alpha_n$. If $U\subset \mathbb{R}^n$ is open and bounded, we define
\[\|f\|_{C^k(\overline{U},Y)}:=\max_{|\alpha|\le k}\sup_{x\in\overline{U}}\|\partial^\alpha_xf(x)\|_Y\]
for all $f\in C^k(\overline{U},Y)$.  If $(Y,\|.\|_Y)$ is a Banach space, then $(C^k(\overline{U},Y),\|.\|_{C^k(\overline{U},Y)})$ is a Banach space.

The following technical lemma will be helpful to show e.g. that the maps that will later be the charts of $C^k(M,N)$ are homeomorphisms.
\begin{lemma}\label{lemma Absch Ck norm verknuepfung} Let $U_1\subset\mathbb{R}^n$ and $W\subset\mathbb{R}^m$ be open. Let $K\subset U_1$  and $\tilde{K}\subset W$ be compact. Let $\Psi\colon W\rightarrow \mathbb{R}^l$ be a $C^k$-map and $R>0$. Let $f_1\in C^k(U_1,\mathbb{R}^m)$ with $f_1(K)\subset\tilde{K}$ and $f_1(U_1)\subset W$. Then there exists $C=C(\Psi, K,\tilde{K},R,f_1)>0$ s.t.
\begin{align*}
&\max_{|\alpha|\le k}\sup_{x\in K}\|\partial^\alpha_x(\Psi\circ f_1)(x)-\partial^\alpha_x(\Psi\circ f_2)(x)\|\\
&\le C\max_{|\alpha|\le k}\sup_{x\in K}\|\partial^\alpha_xf_1(x)-\partial^\alpha_xf_2(x)\|
\end{align*}
for all $f_2\in C^k(U_2,\mathbb{R}^m)$ with $f_2(K)\subset\tilde{K}$, $f_2(U_2)\subset W$, and
\begin{align}\label{eq qqq}
	\max_{|\alpha|\le k}\sup_{x\in K}\|\partial^\alpha_xf_1(x)-\partial^\alpha_x f_2(x)\|\le R.
\end{align} 
Moreover, $C(\Psi, K,\tilde{K},R,f_1)$ can be chosen s.t. $R\mapsto C(\Psi, K,\tilde{K},R,f_1)$ is non-decreasing.
\end{lemma}
\begin{proof}[Sketch of proof.] The assertion of the lemma can be shown by mathematical induction over $k$, applying the chain rule, and adding zeros. We want to illustrate the idea in the case $k=n=m=l=1$ by the following exemplary calculation:
	\begin{align*}
		\partial_x (\Psi\circ f_1)-\partial_x (\Psi\circ f_2)&=(\partial_x\Psi)\circ f_1 \cdot \partial_xf_1-(\partial_x\Psi)\circ f_2 \cdot \partial_xf_2\\
		&=(\partial_x \Psi)\circ f_1 \cdot \big(\partial_x f_1-\partial_x f_2\big) \\
		&\hphantom{=}+ \big((\partial_x\Psi)\circ f_1 - (\partial_x\Psi)\circ f_2\big)\cdot \partial_x f_2.
	\end{align*}
Now we can deal with the terms on the right hand side of the above equation by using the induction hypothesis and \eqref{eq qqq}. For higher differentiability orders and space dimensions, the calculations get more technical, but the idea stays the same. For example, in the case $k=2$ (and $n=m=l=1$) we apply the chain rule to get
\[\partial_x^2(\Psi\circ f_i)= (\partial_x^2\Psi)\circ f_i \cdot (\partial_x f_i)^2 + (\partial_x\Psi)\circ f_i \cdot \partial_x^2f_i.\]
Using this equation and adding zero, we have
\begin{align*}
	\partial_x^2(\Psi\circ f_1)-\partial_x^2(\Psi\circ f_2)=&\big((\partial_x^2\Psi)\circ f_1-(\partial_x^2\Psi)\circ f_2 \big)\cdot (\partial_xf_2)^2\\
	&+(\partial_x^2\Psi)\circ f_1 \cdot \big((\partial_x f_1)^2-(\partial_xf_2)^2\big)\\
	&+\big((\partial_x\Psi)\circ f_1 -(\partial_x\Psi)\circ f_2 \big)\cdot \partial_x^2f_1\\
	&+ (\partial_x\Psi)\circ f_2 \cdot \big(\partial_x^2f_1-\partial_x^2f_2\big).
\end{align*}
Again, we now use the induction hypothesis and \eqref{eq qqq}. (The term $(\partial_x f_1)^2-(\partial_xf_2)^2$ can be dealt with by the binomial formula and \eqref{eq qqq}.)
\end{proof}

The goal  for the remainder of this section is to state and prove the so-called (local) $\Omega$-lemma. As stated in the introduction, this lemma is the key to show that $C^k(M,N)$ carries a smooth structure. To that end, we recall the following version of Taylor's theorem. 

Suppose that $X$ is a Banach space and that $U\subset X$ is an open convex subset. An open subset $\tilde{U}\subset X\times X$ is \textit{a thickening of $U$} if
\begin{enumerate}
	\item $U\times \{0\}\subset \tilde{U}$,
	\item $u+th\in U$ for all $(u,h)\in \tilde{U}$ and $0\le t\le 1$,
	\item $(u,h)\in\tilde{U}$ implies $u\in U$.
\end{enumerate}
Note that there always exists a thickening of $U$.
\begin{lemma}[Taylor's theorem]\label{theorem taylor} Let $X$ and $Y$ be Banach spaces, $U\subset X$ open and convex,  $\tilde{U}$ a thickening of $U$. A map $f\colon U\rightarrow Y$ is $r$ times continuously differentiable if and only if there are continuous maps
\[\varphi_i\colon U\rightarrow L^i_s(X,Y), \hspace{3em} i=1,\ldots r,\]
and
\[R\colon\tilde{U}\rightarrow L^r_s(X,Y),\]
s.t. for all $(u,h)\in \tilde{U}$,
\[f(u+h)=f(u)+\left(\sum_{i=1}^r\frac{\varphi_i(u)}{i!}h^i\right)+R(u,h)h^r\]
where $h^i=(h,\ldots,h)$ ($i$ times) and $R(u,0)=0$. If $f$ is $r$ times continuously differentiable, then necessarily $\varphi_i=D^if$ for all $i=1,\ldots, r$ and in addition
\[R(u,h)=\int_{0}^{1}\frac{(1-t)^{r-1}}{(r-1)!}\left(D^rf(u+th)-D^rf(u) \right)dt\]
\end{lemma}
A proof can be found in e.g. \cite[2.4.15 Theorem]{MTAA}.

\begin{lemma}[local $\Omega$-lemma]\label{lemma loc omega}Let $r,l\in\mathbb{N}$. Let $U\subset\mathbb{R}^n$ be open and bounded and let $V\subset\mathbb{R}^m$ be open, bounded, and convex. Moreover, let $Y$ be a Banach space and 
	\[g\colon U\times V\rightarrow Y\]
	a map s.t.
	\begin{enumerate}
		\item $g\in C^r(\overline{U\times V},Y).$
		\item For each $i\in\{0,\ldots,l\}$ the map
		\[D^i_2g\colon U\times V\rightarrow L^i_s(\mathbb{R}^m,Y),\]
		defined by $(D^i_2g)(x,y):=(D^i(g(x,.))(y)$ for all $(x,y)\in U\times V$ exists and is an element of $C^r(\overline{U\times V},L^i_s(\mathbb{R}^m,Y)).$
	\end{enumerate}
	Then the map
	\begin{align*}
	\Omega_g\colon C^r(\overline{U},V)&\rightarrow C^r(\overline{U},Y)\\
	f&\mapsto (x\mapsto g(x,f(x)))
	\end{align*}
	is an element of $C^l(C^r(\overline{U},V),C^r(\overline{U},Y))$. Here,
	\[C^r(\overline{U},V):=\{f\in C^r(\overline{U},\mathbb{R}^m)\text{ }|\text{ } f(\overline{U})\subset V\}\]
	and $C^r(\overline{U},V) \subset C^r(\overline{U},\mathbb{R}^m)$ is open. Moreover, if $l>0$, it holds that
	\begin{align}\label{eq4}
		D^i\left(\Omega_g\right)=A_i\circ \Omega_{D^i_2g}
	\end{align}
	for each $i=1,\ldots,l$, where $A_i$ is the continuous map
	\[A_i\colon  C^r(\overline{U},L^i_s(\mathbb{R}^m,Y))\rightarrow L^i_s(C^r(\overline{U},\mathbb{R}^m),C^r(\overline{U},Y))\]
	defined by
	\[\left(\left(A_i(H)\right)(h_1,\ldots,h_i)\right)(x):=(H(x))(h_1(x),\ldots,h_i(x))\]
\end{lemma}
The statement of Lemma \ref{lemma loc omega} can be found in different versions in \cite{Blue,MTAA, Gl}. We want to humbly point out that it is possible that \cite[3.6 Theorem]{Blue} only holds in special cases. This theorem is tied to the assumptions of the version of the local $\Omega$-lemma \cite[3.7 Theorem]{Blue}. Therefore it is possible that the assumptions of the local $\Omega$-lemma in \cite{Blue} are not ideal.

Our proof is an adapted version of \cite[Proof of 2.4.18 Proposition]{MTAA}.
\begin{proof}[Proof of Lemma \ref{lemma loc omega}] First we prove that $C^r(\overline{U},V) \subset C^r(\overline{U},\mathbb{R}^m)$ is open. Choose $f_0\in C^r(\overline{U},V)$. Since $f_0(\overline{U})$ is compact, $\mathbb{R}^m\setminus V$ is closed, and $f_0(\overline{U})\cap(\mathbb{R}^m\setminus V)=\varnothing$, we have
	\[\varepsilon :=\textup{dist}(f_0(\overline{U}),\mathbb{R}^m\setminus V)>0.\]
Now assume that $\|f-f_0\|_{C^r(\overline{U},\mathbb{R}^m)}<\varepsilon$. It follows that $\|f(x)-f_0(x)\|_Y<\varepsilon$ for all $x\in\overline{U}$. By definition of $\varepsilon$, this means $f(\overline{U})\subset V$ and so $C^r(\overline{U},V) \subset C^r(\overline{U},\mathbb{R}^m)$ is open.

In the case ``$l=0$, $r\in\mathbb{N}$'' the assertion of the lemma follows from a computation. Assume $l\in\mathbb{N}_{>0}$ and $r\in\mathbb{N}$. Let $\tilde{V}\subset \mathbb{R}^m\times\mathbb{R}^m$ be a thickening of $V$. From applying Lemma \ref{theorem taylor} to $g(x,.)$ (for $x$ fixed) it follows that for all $(y_1,y_2)\in \tilde{V}$ and all $x\in U$ we have
\begin{align}\label{eq3}
g(x,y_1+y_2)=g(x,y_1)+\left(\sum_{i=1}^{l}\frac{1}{i!}(D^i_2g)(x,y_1)y_2^i\right)+R(x,y_1,y_2)y_2^l
\end{align}
where the map
\[R\colon U\times\tilde{V} \rightarrow L^l_s(\mathbb{R}^m,Y)\]
is given by
\[R(x,y_1,y_2)=\int_{0}^{1}\frac{(1-t)^{l-1}}{(l-1)!}\left(D^l_2g(x,y_1+ty_2)-D^l_2g(x,y_1) \right)dt.\]
Define \[F(t,x,y_1,y_2):=\frac{(1-t)^{l-1}}{(l-1)!}\left(D^l_2g(x,y_1+ty_2)-D^l_2g(x,y_1)\right).\]
From ii) it follows that
\[F\in C^r(\overline{(0,1)\times U\times\tilde{V}}, L^l_s(\mathbb{R}^m,Y)).\]
By differentiating under the integral it follows that
\[R\in C^r(\overline{U\times\tilde{V}},L^l_s(\mathbb{R}^m,Y)).\]
Since we already proved the case ``$l=0$, $r\in\mathbb{N}$'' we see that
\begin{align*}
\Omega_R\colon C^r(\overline{U},\tilde{V})&\rightarrow C^r(\overline{U},L^l_s(\mathbb{R}^m,Y)),\\
h&\mapsto (x\mapsto R(x,h(x))),
\end{align*}
is continuous. In particular,
\[\tilde{R}:=A_l\circ\Omega_R\colon C^r(\overline{U},\tilde{V})\rightarrow L^l_s(C^r(\overline{U},\mathbb{R}^m),C^r(\overline{U},Y))\]
is continuous. Analogously, we see that
\[\widetilde{\Omega_{D^i_2g}}:=A_i\circ\Omega_{D^i_2g}\colon C^r(\overline{U},V)\rightarrow L^i_s(C^r(\overline{U},\mathbb{R}^m),C^r(\overline{U},Y))\]
is continuous for $i=1,\ldots,l$.
From $\eqref{eq3}$ it follows that for all $(f,h)\in C^r(\overline{U},\tilde{V})$ we have
\[\Omega_g(f+h)=\Omega_g(f)+\left(\sum_{i=1}^{l} \frac{1}{i!}\widetilde{\Omega_{D^i_2g}}(f)h^i\right)+\tilde{R}(f,h)h^l.\]
From Lemma \ref{theorem taylor} we conclude that $\Omega_g\in C^l(C^r(\overline{U},V),C^r(\overline{U},Y))$ and
\[D^i\left(\Omega_g\right)=\widetilde{\Omega_{D^i_2g}}=A_i\circ \Omega_{D^i_2g}\]
for $i=1,\ldots,l$. (Here we used that $C^r(\overline{U},\tilde{V})$, viewed as a subset of $C^r(\overline{U},\mathbb{R}^m)\times C^r(\overline{U},\mathbb{R}^m)$, is a thickening of $C^r(\overline{U},V)$.)
\end{proof}

\section{The topological space $C^k(M,N)$}
In this section we recall the definitions of the compact-open $C^k$ topology on $C^k(M,N)$ and the $C^k$-norm on sections of vector bundles. We try to be as precise as possible when stating these definitions, so that no confusion arises when we use them later in technical proofs. Then we show that the maps which will be the charts of $C^k(M,N)$ are homeomorphisms.

The following definition is taken from \cite[Chapter 2]{Hir}.

\begin{definition}[compact-open $C^k$ topology] Let $M$ and $N$ be manifolds without boundary and $k\in\mathbb{N}$. For $f\in C^k(M,N)$, charts $(\varphi,U)$ and $(\psi,V)$ of $M$ and $N$, respectively, $K\subset U$ compact with $f(K)\subset V$ and $\varepsilon >0$ we define the set
\begin{align*}
\mathcal{N}^k(f,\varphi,U,\psi&,V,K,\varepsilon):=\{g\in C^k(M,N)\text{ }|\text{ }g(K)\subset V\text{, }\\
&\max_{|\alpha|\le k}\sup_{x\in\varphi(K)}\|\partial^\alpha_x(\psi\circ g\circ \varphi^{-1})(x)-\partial^\alpha_x(\psi\circ f\circ \varphi^{-1})(x)\|< \varepsilon\}
\end{align*}
where $\|.\|$ denotes the Euclidean norm. The  \textit{compact-open $C^k$ topology} (or \textit{weak topology}) \textit{on $C^k(M,N)$} is the topology generated by the set
\begin{align*}
\{\mathcal{N}^k(f,\varphi,U,\psi,V,K,\varepsilon)\text{ }|\text{ }f\in C^k(M,N),\text{ } (\varphi,U) \text{ and }(\psi,V) \text{ charts of } M \text{ and } N,\\
\text{ respectively, } K\subset U \text{ compact with } f(K)\subset V \text{, } \varepsilon >0\}
\end{align*}
as a subbasis.
\end{definition}
From now on, we always assume $C^k(M,N)$ to be equipped with the compact-open $C^k$ topology. The topological space $C^k(M,N)$ is secound-countable and metrizable \cite[p. 35]{Hir}. In particular, it is Hausdorff.

We will use the following lemma later.
\begin{lemma}\label{lemma weak top properties} Assume $M$ is closed. Let $f\in C^k(M,N)$, $k\in\mathbb{N}$, $(\varphi_i,U_i)$ and $(\psi_i,V_i)$ charts of $M$ and $N$ respectively, $K_i\subset U_i$ compact with $f(K_i)\subset V_i$, $i=1,\ldots r$, and $\bigcup_{i=1}^rK_i=M$. Then the set 
		\begin{align*}
				\{\bigcap_{i=1}^r\mathcal{N}^k(f,\varphi_i,U_i,\psi_i,V_i,K_i,\varepsilon)\text{ }|\text{ }\varepsilon >0\}
		\end{align*}
		is a neighborhood basis of $f$.
		In particular, a sequence $(f_m)_{m\in\mathbb{N}}\subset C^k(M,N)$ converges to $f$ in $C^k(M,N)$ iff for all $\varepsilon>0$ there exists some $N=N(\varepsilon)$ s.t. for all $m\ge N(\varepsilon)$ it holds that $f_m\in \bigcap_{i=1}^r\mathcal{N}^k(f,\varphi_i,U_i,\psi_i,V_i,K_i,\varepsilon)$.
\end{lemma}
\begin{proof} We have to show the following: If an arbitrary $\mathcal{N}^k(f,\varphi,U,\psi,V,K,\varepsilon)$ is given, then there exists some $\delta>0$ s.t. 
	\[\bigcap_{i=1}^r\mathcal{N}^k(f,\varphi_i,U_i,\psi_i,V_i,K_i,\delta)\subset\mathcal{N}^k(f,\varphi,U,\psi,V,K,\varepsilon).\]
	To that end, assume that $K_i\cap K\neq\varnothing$. Since the complement $\psi_i(V_i\cap V)^\complement$ is closed, $\psi_i(f(K_i\cap K))$ is compact, and $\psi_i(V_i\cap V)^\complement\cap\psi_i(f(K_i\cap K))=\varnothing$ we have
	\[\delta_i:=\mathrm{dist}(\psi_i(V_i\cap V)^\complement,\psi_i(f(K_i\cap K)))>0.\]
	Now choose an arbitrary $\delta$ with
	\[0<\delta\le\frac{1}{2}\min\{\delta_i\text{ }|\text{ } i\in\{1,\ldots,r\} \text{ and } K_i\cap K\neq \varnothing\}\]
	and let \[g\in\bigcap_{i=1}^r\mathcal{N}^k(f,\varphi_i,U_i,\psi_i,V_i,K_i,\delta).\] We show $g(K)\subset V$. Since $g(K_i)\subset V_i$ and because the $K_i$ cover $M$, it is sufficient to show $g(K_i\cap K)\subset V_i\cap V$ whenever $K_i\cap K\neq \varnothing$.
	To that end, assume $K_i\cap K\neq \varnothing$. From $g\in \mathcal{N}^k(f,\varphi_i,U_i,\psi_i,V_i,K_i,\delta)$ it follows that
	\[\max_{|\alpha|\le k}\sup_{x\in\varphi_i(K_i\cap K)}\|\partial^\alpha_x(\psi_i\circ g\circ \varphi^{-1})(x)-\partial^\alpha_x(\psi_i\circ f \circ \varphi^{-1})(x)\|<\delta.\]
	In particular, that means that for each $p\in K_i\cap K$ we have $\psi_i(g(p))\in B_{\delta}(\psi_i(f(p)))$. From the definition of $\delta$ it follows that for all $p\in K_i\cap K$ we have $B_{\delta}(\psi_i(f(p)))\subset\psi_i(V_i\cap V)$. It follows that $\psi_i(g(K_i\cap K))\subset \psi_i(V_i\cap V)$ and thus $g(K_i\cap K))\subset V_i\cap V$. We have shown $g(K)\subset V$. Using Lemma \ref{lemma Absch Ck norm verknuepfung}\footnote{For $f_1=\psi_i\circ f \circ \varphi_i^{-1}$ defined on $\varphi_i(U_i\cap U\cap f^{-1}(V_i\cap V))$,  $f_2=\psi_i\circ g\circ\varphi_i^{-1}$ defined on $\varphi_i(U_i\cap U\cap g^{-1}(V_i\cap V))$, $\Psi=\psi\circ \psi_i^{-1}$ defined on $\psi_i(V_i\cap V)$, and $\tilde{K}=\overline{ B_\delta(\psi_i(f(K_i\cap K)))}\subset \psi_i(V_i\cap V)$.} (and a version of Lemma \ref{lemma Absch Ck norm verknuepfung} that estimates pre-composition with diffeomorphisms rather than post-composition with maps, for details see \cite[Lemma 3.2.1 i)]{JWDissertation}) we calculate
	\begin{align*}
	&\max_{|\alpha|\le k}\sup_{x\in\varphi(K)}\|\partial^\alpha_x(\psi\circ g\circ \varphi^{-1})(x)-\partial^\alpha_x(\psi\circ f\circ \varphi^{-1})(x)\|\\
	&=\max_{i=1,\ldots,l}\max_{|\alpha|\le k}\sup_{x\in\varphi(K_i\cap K)}\|\partial^\alpha_x(\psi\circ g\circ \varphi^{-1})(x)-\partial^\alpha_x(\psi\circ f\circ \varphi^{-1})(x)\|\\
	&=\max_{i=1,\ldots,l}\max_{|\alpha|\le k}\sup_{x\in\varphi(K_i\cap K)}\|\partial^\alpha_x(\psi\circ \psi_i^{-1}\circ \psi_i\circ g\circ \varphi_i^{-1}\circ\varphi_i\circ \varphi^{-1})(x)\\
	&\hspace{3cm}-\partial^\alpha_x(\psi\circ\psi_i^{-1}\circ \psi_i\circ f\circ\varphi_i^{-1}\circ\varphi_i\circ \varphi^{-1})(x)\|\\
	&\le\max_{i=1,\ldots,l}\left(C_i\max_{|\alpha|\le k}\sup_{x\in\varphi_i(K_i\cap K)}\|\partial^\alpha_x(\psi_i\circ g\circ \varphi_i^{-1})(x)-\partial^\alpha_x(\psi_i\circ f\circ \varphi_i^{-1})(x)\|\right)\\
	&\le \left(\max_{i=1,\ldots,l}C_i\right)\delta.
	\end{align*}
	Now we choose $\delta$ so small that $\left(\max_{i=1,\ldots,l}C_i\right)\delta<\varepsilon$. This finishes the proof.
\end{proof}

\begin{definition}[$C^k$-norm on sections of a vector bundle]\label{def C^k norm vrb} Let $M$ be a closed manifold. Let $\pi\colon E\to M$ be a $C^k$ vector bundle. Pick charts $(U_i,\varphi_i)$ of $M$, $i=1,\ldots,l$, $\bigcup_{i=1}^lU_i=M$ s.t. $\overline{U_i}\subset M$ is compact, $\overline{U_i}\subset \tilde{U_i}$,  $(\tilde{U_i},\varphi_i)$ is still a chart of $M$ and there are local trivializations $(\hat{U}_i,\Phi_i)$ of $E$ with $\overline{U_i}\subset\hat{U}_i$ for each $i=1,\ldots,l$. For $k\in\mathbb{N}$ let
\[\Gamma_{C^k}(E):=\{s\colon M \to E\text{ }|\text{ } s\in C^k(M,E) \text{  and } \pi\circ s=id_M \}\]
be the space of $C^k$-sections of $E$. Define the \textit{$C^k$-norm on $\Gamma_{C^k}(E)$} by
\[\|s\|_{C^k}:=\|s\|_{\Gamma_{C^k}(E)}:=\max_{i=1,\ldots, l}\max_{|\alpha|\le k} \sup_{x\in\overline{\varphi_i(U_i)}}\|\partial^\alpha_x(pr_2\circ\Phi_i\circ s\circ \varphi_i^{-1})\|\]
for $s\in\Gamma_{C^k}(E)$.

Note that $(\Gamma_{C^k}(E),\|.\|_{C^k})$ is a Banach space. Up to equivalence of norms, $\|.\|_{C^k}$ does not depend on the choices made in its definition.

\end{definition}

For the definition of the charts of $C^k(M,N)$ the exponential map of $N$ is the main input. For the convenience of the reader and to fix notation we recall some basic facts about the exponential map of a Riemannian manifold. In the following we denote the bundle projection of $TN$ by $\pi_{TN}\colon TN\to N$.

\begin{lemma}\label{lemma eig exp}Let $(N,h)$ be a Riemannian manifold. Define $\mathcal{E}\subset TN$ by
\[\mathcal{E}:=\{v\in TN\text{ }|\text{ } \textup{exp}_{\pi_{TN}(v)}v \text{ exists}\}.\]
\begin{enumerate}
\item $\mathcal{E}\subset TN$ is open and 
\[\textup{exp}\colon \mathcal{E}\rightarrow N\]
defined by $\textup{exp}(v):=\textup{exp}_{\pi_{TN}(v)}v$ is smooth.
\item  Define the smooth map
\[E:=(\pi_{TN},\textup{exp})\colon \mathcal{E}\rightarrow N\times N\]
by $E(v):=(\pi_{TN}(v),\textup{exp}_{\pi_{TN}(v)}v)$. For each $p\in N$ there exists a neighborhood $W$ of $0_p$ (where $0_p$ denotes the zero-element in $T_pN$) in $TN$ s.t. the map
\[E\colon W\rightarrow E(W)\]
is a diffeomorphism (in particular $E(W)$ is open in $N\times N$).
\item For all $p\in N$ and  $0<\delta<\textup{inj}_p(N)$ where $\textup{inj}_p(N)>0$ is the injectivity radius of $N$ at $p$ it holds that
\[\textup{exp}_p\colon B_{\delta}(0_p)\rightarrow B_\delta(p)\]
is a diffeomorphism where $B_{\delta}(0_p)=\{v\in T_pM\text{ }|\text{ }\|v\|_h:=\delta \}$, $B_\delta(p)=\{q\in N\text{ }|\text{ }d(p,q)<\delta\}$, and $d$ is the distance function induced by $h$.
\end{enumerate}
\end{lemma}

Now we define the maps that will later be the charts of $C^k(M,N)$ and show that they are homeomorphisms.
\begin{lemma}\label{lemma charts c^kMN are homeo}Let $k\in\mathbb{N}$. Let $M$ and $N$ be manifolds without boundary. Let $M$ be compact and let $N$ be connected. Choose a Riemannian metric $h$ on $N$. Define
	\[U_{f,\varepsilon}:=\bigcap_{i=1}^l\mathcal{N}^k(f,\varphi_i,\tilde{U}_i,\psi_i,V_i,\overline{U_i},\varepsilon)\]
for $(U_i,\varphi_i)$ charts of $M$, $i=1,\ldots,l$, $\bigcup_{i=1}^lU_i=M$, s.t. $\overline{U_i}\subset M$ is compact, $\overline{U_i}\subset\tilde{U}_i$, $(\tilde{U}_i,\varphi_i)$ is still a chart of $M$ and charts $(V_i,\psi_i)$ of $N$ with $f(\overline{U_i})\subset V_i$ for each $i=1,\ldots,l$, $\varepsilon>0$.
Define the map
\[\varphi_f\colon U_{f,\varepsilon}\rightarrow \varphi_f(U_{f,\varepsilon})\subset \Gamma_{C^k}(f^*TN)\]
by
\[(\varphi_f(g))(p):=(\textup{exp}_{f(p)})^{-1}g(p)\]
for all $p\in M$, where $\textup{exp}$ is the exponential map of $(N,h)$. Then it holds that
\begin{enumerate}
	\item For every $\delta>0$ there exists $\varepsilon >0$ s.t. for all $g\in U_{f,\varepsilon}$ and all $p\in M$ we have 
	\begin{align*}
	d(g(p),f(p))<\delta.
	\end{align*}
	In particular, $\varphi_f$ is well-defined on $U_{f,\varepsilon(\delta)}$ for $\delta< \inf_{p\in M}\textup{inj}_{f(p)}(N)$.
\end{enumerate}
Moreover, for $\varepsilon>0$ small enough the following is true:
\begin{enumerate}
	\item[ii)] $\varphi_f\colon U_{f,\varepsilon}\rightarrow \varphi_f(U_{f,\varepsilon})$ is continuous (where on $U_{f,\varepsilon}$ we have the subspace topology induced from the compact-open $C^k$ topology and on $\varphi_f(U_{f,\varepsilon})$ we have the subspace topology induced from the $C^k$-norm on $\Gamma_{C^k}(f^*TN)$).
	\item[iii)] $\varphi_f(U_{f,\varepsilon})\subset \Gamma_{C^k}(f^*TN)$ is open. 
	\item[iv)] $\varphi_f^{-1}\colon \varphi_f(U_{f,\varepsilon})\rightarrow U_{f,\varepsilon}$ is continuous.
\end{enumerate}
\end{lemma}
\begin{proof} We start by mentioning that since $C^k(M,N)$ and $\Gamma_{C^k}(f^*TN)$ are first-countable, it is sufficient to show that $\varphi_f$ and $\varphi_f^{-1}$ are sequentially continuous. To make the proofs of i) and ii) easier, we first choose the $V_i$ s.t.
	\[(A)\left\{ \begin{array}{l}  \psi_i(V_i) \text{ is convex and compact, }\overline{V_i}\subset\tilde{V}_i \text{ where } (\tilde{V}_i,\psi_i) \text{ is still a chart of }N\text{,}\\\tilde{V}_i\times \tilde{V}_i\subset E(W_i), \text{where }W_i\subset TN\text{ and}\\ E(W_i)\subset N\times N\text{ are open s.t. } \\E\colon W_i\rightarrow E(W_i) \text{ is a diffeomorphism},\\ (\tilde{V}_i, \hat{\Phi}_i) \text{ are local trivializations of } TN \text{  with induced local}\\\text{trivialization } (f^{-1}(\tilde{V_i}),\Phi_i) \text{  of } f^*TN \text{ for each } i=1,\ldots l. \end{array} \right.\]
	(See Lemma \ref{lemma eig exp} ii).)

In the following we prove i) and ii) with the additional assumption $(A)$ and then show afterwards that we don't need it, provided that $\varepsilon>0$ is small enough.

\textbf{Proof of i):} It is not difficult to see that for every $\delta>0$ there exists $\varepsilon>0$ s.t. for all $g\in U_{f,\varepsilon}$ and all $p\in M$ we have
\[d(g(p),f(p))<\delta.\]
Choosing $\delta<\inf_{p\in M}\textup{inj}_{f(p)}(N)$ we have that $(\textup{exp}_{f(p)})^{-1}g(p)$ exists for each $p\in M$ and all $g\in U_{f,\varepsilon(\delta)}$. Moreover, $\varphi_f(g)\in\Gamma_{C^k}(f^*TN)$, since on $U_i$ it holds that $\varphi_f(g)=(E|_{W_i})^{-1}\circ(f,g)$. We have shown that $\varphi_f$ is a well-defined map.

\textbf{Proof of ii):} Choose $\varepsilon$ so small that $\varphi_f$ is well-defined on $U_{f,\varepsilon}$, see i). Let $(g_m)_{m\in \mathbb{N}}$ be a sequence in $U_{f,\varepsilon}$, $g\in U_{f,\varepsilon}$ with $g_m\xrightarrow{m\to\infty}g$ in $U_{f,\varepsilon}$. In particular, for each $r>0$ there exists $N=N(r)\in\mathbb{N}$ s.t. \[g_m\in\bigcap_{i=1}^l\mathcal{N}^k(g,\varphi_i,\tilde{U}_i,\psi_i,V_i,\overline{U_i},r)\]
for all $m\ge N$. (We note that the $\varphi_i,\tilde{U}_i,\psi_i,V_i,U_i$ are the same as in the statement of the lemma where we additionally assume $(A)$ as mentioned above.) That means, that for all $i=1,\ldots,l$ we have
\[\|\psi_i\circ g_m\circ \varphi_i^{-1}-\psi_i\circ g\circ \varphi_i^{-1}\|_{C^k(\overline{\varphi_i(U_i)},\mathbb{R}^n)}\xrightarrow{m\to\infty}0\]
where $n=\textup{dim}(N)$. Using Lemma \ref{lemma Absch Ck norm verknuepfung} 
\footnote{For $f_1=(\psi_i\times \psi_i)\circ (f,g)\circ \varphi_i^{-1}$ defined on $\varphi_i(\tilde{U}_i\cap f^{-1}(\tilde{V}_i)\cap g^{-1}(\tilde{V}_i))$, $f_2=(\psi_i\times \psi_i)\circ (f,g_m)\circ \varphi_i^{-1}$ defined on $\varphi_i(\tilde{U}_i\cap f^{-1}(\tilde{V}_i)\cap g_m^{-1}(\tilde{V}_i))$, and for $\Psi=pr_2\circ\hat{\Phi}_i\circ E|_{W_i}^{-1}\circ(\psi_i^{-1}\times\psi_i^{-1})$, defined on $\psi_i(\tilde{V}_i)\times\psi_i(\tilde{V}_i)$, $K=\overline{\varphi_i(U_i)}$, and $\tilde{K}=\psi_i(\overline{V_i})\times\psi_i(\overline{V_i})$.} we calculate for each $i=1,\ldots,l$
\begin{align*}
&\|pr_2\circ\Phi_i\circ\varphi_f(g_m)\circ\varphi_i^{-1}-pr_2\circ\Phi_i\circ\varphi_f(g)\circ\varphi_i^{-1}\|_{C^k(\overline{\varphi_i(U_i)},\mathbb{R}^n)}\\
&=\|pr_2\circ\hat{\Phi}_i\circ E|_{W_i}^{-1}\circ(f,g_m)\circ\varphi_i^{-1}\\
&\hphantom{=}-pr_2\circ\hat{\Phi}_i\circ E|_{W_i}^{-1}\circ(f,g)\circ\varphi_i^{-1}\|_{C^k(\overline{\varphi_i(U_i)},\mathbb{R}^n)}\\
&=\|\left(pr_2\circ\hat{\Phi}_i\circ E|_{W_i}^{-1}\circ(\psi_i^{-1}\times\psi_i^{-1})\right)\circ\left((\psi_i\times\psi_i)\circ(f,g_m)\circ\varphi_i^{-1}\right)\\
&\hphantom{=}-\left(pr_2\circ\hat{\Phi}_i\circ E|_{W_i}^{-1}\circ(\psi_i^{-1}\times\psi_i^{-1})\right)\circ\left((\psi_i\times\psi_i)\circ(f,g)\circ\varphi_i^{-1}\right)\|_{C^k(\overline{\varphi_i(U_i)},\mathbb{R}^n)}\\
&\le C_i\|(\psi_i\times\psi_i)\circ(f,g_m)\circ\varphi_i^{-1}-(\psi_i\times\psi_i)\circ(f,g)\circ\varphi_i^{-1}\|_{C^k(\overline{\varphi_i(U_i)},\mathbb{R}^n\times\mathbb{R}^n)}\\
&=C_i\|\psi_i\circ g_m\circ\varphi_i^{-1}-\psi_i\circ g\circ\varphi_i^{-1}\|_{C^k(\overline{\varphi_i(U_i)},\mathbb{R}^n)}\xrightarrow{n\to\infty}0,
\end{align*}
i.e.,
\[\|\varphi_f(g_m)-\varphi_f(g)\|_{C^k}\xrightarrow{m\to\infty}0.\]
Hence, $\varphi_f\colon U_{f,\varepsilon}\rightarrow \varphi_f(U_{f,\varepsilon})$ is continuous.

We have shown i) and ii) under the additional assumption $(A)$. Now we show that we don't need the assumption $(A)$, provided that $\varepsilon>0$ is small enough. To that end, choose $(U_i',\varphi_i')$ charts of $M$, $i=1,\ldots,m$, $\bigcup_{i=1}^mU_i'=M$, s.t. $\overline{U_i'}\subset M$ is compact, $\overline{U_i'}\subset\tilde{U}_i'$, $(\tilde{U}_i',\varphi_i')$ is still a chart of $M$ and charts $(V_i',\psi_i')$ of $N$ with $f(\overline{U_i'})\subset V_i'$ for each $i=1,\ldots,m$. Using Lemma \ref{lemma weak top properties} we choose $\varepsilon'>0$ s.t.
\[\bigcap_{i=1}^m\mathcal{N}^k(f,\varphi_i',\tilde{U}_i',\psi_i',V_i',\overline{U_i}',\varepsilon')\subset \bigcap_{i=1}^l\mathcal{N}^k(f,\varphi_i,\tilde{U}_i,\psi_i,V_i,\overline{U_i},\varepsilon)\]
where the $\varphi_i,\tilde{U}_i,\psi_i,V_i,U_i$ are the same as in the statement of the lemma and satisfy $(A)$.
Since $\varphi_f$ is well-defined and continuous on the set $\bigcap_{i=1}^l\mathcal{N}^k(f,\varphi_i,\tilde{U}_i,\psi_i,V_i,\overline{U_i},\varepsilon)$ (that is what we have shown above) it is obviously well-defined and continuous on the subset $\bigcap_{i=1}^m\mathcal{N}^k(f,\varphi_i',\tilde{U}_i',\psi_i',V_i',\overline{U_i}',\varepsilon')$.

\textbf{Proof of iii) and iv):} Choose $\varepsilon$ so small that $\varphi_f$ is well-defined on $U_{f,\varepsilon}$ and ii) holds. From Lemma \ref{lemma eig exp} iii) we see that $\varphi_f(U_{f,\varepsilon})\subset U:=\{s\in\Gamma_{C^k}(f^*TN)\text{ }|\text{ }\|s(p)\|_h<\delta\text{  for all }p\in M\}$. First we prove that $U$ is open in $\Gamma_{C^k}(f^*TN)$. To that end, let $s_0\in U$. Since the function $M\rightarrow \mathbb{R}$, $p\mapsto \|s_0(p)\|_h$, is continuous and $M$ is compact, we have $\delta_0:=\max_{p\in M}\|s_0(p)\|_h<\delta$. Comparing $h$ to the Euclidean norm in the trivialization it is easy to verify that there exists $C>0$ s.t.
\[\|s(p)-s_0(p)\|_h\le C\|s-s_0\|_{C^k}\]
for all $s\in\Gamma_{C^k}(f^*TN)$ and all $p\in M$. Choose $r>0$ s.t. $Cr<\delta-\delta_0$. If $\|s-s_0\|_{C^k}<r$, then 
\[\|s(p)\|_h\le \|s(p)-s_0(p)\|_h+\|s_0(p)\|_h\le C\|s-s_0\|_{C^k} +\delta_0 < Cr+\delta_0<\delta\]
for all $p\in M$, therefore $U$ is open in $\Gamma_{C^k}(f^*TN)$.

Next we show that the well-defined map
\[H\colon U\rightarrow C^k(M,N),\]
$(H(s))(p):=\textup{exp}_{f(p)}s(p)$ is continuous. Then we have in particular that $\varphi_f^{-1}=H|_{\varphi_f(U_{f,\varepsilon})}$ is continuous and that $\varphi_f(U_{f,\varepsilon})=H^{-1}(U_f)$ is open in $U$ (and therefore also in $\Gamma_{C^k}(f^*TN)$).

To show continuity of $H$, choose charts $(U_i,\varphi_i)$ of $M$, $i=1,\ldots,l$, $\bigcup_{i=1}^lU_i=M$ s.t. $\overline{U_i}\subset M$ is compact, $\overline{U_i}\subset \tilde{U_i}$,  $(\tilde{U_i},\varphi_i)$ is still a chart of $M$ and there are local trivializations $(\tilde{U}_i,\Phi_i)$ of $f^*TN$ and charts $(V_i,\psi_i)$ of $N$ with $f(\overline{U_i})\subset V_i$ and $(B_{\delta}(V_i),\psi_i)$ is still a chart of $N$ for each $i=1,\ldots,l$, where $B_{\delta}(V_i)=\{p\in N\text{ }|\text{ } \exists q\in V_i:\text{ }d(p,q)<\delta\}$. (Note that the $\varphi_i,\tilde{U}_i,\psi_i,V_i,U_i$ here don't need to be the same as in the statement of the lemma.)

Let $(s_m)_{m\in \mathbb{N}}$ be a sequence in $U$, $s\in U$, with 
\[\|s_m-s\|_{C^k}\xrightarrow{m\to\infty}0,\]
i.e.,
\begin{align*}
\|pr_2\circ \Phi_i\circ s_m\circ\varphi_i^{-1}-pr_2\circ\Phi_i\circ s\circ\varphi_i^{-1}\|_{C^k(\overline{\varphi_i(U_i)},\mathbb{R}^n)}\xrightarrow{m\to\infty}0
\end{align*}
for each $i=1,\ldots,l$. For showing $H(s_m)\xrightarrow{m\to\infty}H(s)$ in $C^k(M,N)$ it is sufficient to show that for all $r >0$ there exists $N=N(r)\in\mathbb{N}$ s.t. 
\[H(s_m)\in\bigcap_{i=1}^l\mathcal{N}^k(H(s),\varphi_i,\tilde{U}_i,\psi_i,B_\delta(V_i),\overline{U_i},r)\] 
for all $m\ge N$, see Lemma \ref{lemma weak top properties}. First of all, by definition of $H$ and Lemma \ref{lemma eig exp} iii) it holds that 
\[d(H(s_m)(p),f(p))<\delta \text{ for all }m\in\mathbb{N}\text{ and }d(H(s)(p),f(p))<\delta\]
for each $p\in M$. Since $f(\overline{U_i})\subset V_i$ it follows that $(H(s_m)(\overline{U_i})\subset B_\delta(V_i)$ and $(H(s)(\overline{U_i})\subset B_\delta(V_i)$ for each $m\in\mathbb{N}$ and $i=1,\ldots,l$. Let $r >0$. Using Lemma \ref{lemma Absch Ck norm verknuepfung} 
\footnote{For $f_1=(\varphi_i\times id)\circ \Phi_i\circ s\circ \varphi_i^{-1}$ defined on $\varphi_i(\tilde{U}_i\cap f^{-1}(B_\delta(V_i))$, $f_2=(\varphi_i\times id)\circ \Phi_i\circ s_m\circ \varphi_i^{-1}$ also defined on $\varphi_i(\tilde{U}_i\cap f^{-1}(B_\delta(V_i)))$,  $\Psi=\psi_i\circ f^*\textup{exp}\circ \Phi_i^{-1}\circ (\varphi_i^{-1}\times id)$ defined on $(\varphi_i\times id)\circ \Phi_i\left(\{v\in f^*TN\text{ }|\text{ }\|v\|<\delta \}\cap f^*TN|_{\tilde{U}_i\cap f^{-1}(B_\delta(V_i))}\right)$, $K=\overline{\varphi_i(U_i)}$, and $\tilde{K}=(\varphi_i\times id)\circ \Phi_i\left(\{v\in f^*TN\text{ }|\text{ }\|v\|\le\delta \}\cap f^*TN|_{\overline{U}_i\cap f^{-1}(\overline{V_i})}\right)$.}
we calculate for each $i=1,\ldots,l$  and $m$ large enough 
\begin{align*}
&\|\psi_i\circ H(s_m)\circ\varphi_i^{-1}-\psi_i\circ H(s)\circ\varphi_i^{-1}\|_{C^k(\overline{\varphi_i(U_i)},\mathbb{R}^n)}\\
&=\|\psi_i\circ f^*\textup{exp}\circ s_m\circ\varphi_i^{-1}-\psi_i\circ f^*\textup{exp}\circ s\circ\varphi_i^{-1}\|_{C^k(\overline{\varphi_i(U_i)},\mathbb{R}^n)}\\
&=\|\left(\psi_i\circ f^*\textup{exp}\circ\Phi_i^{-1}\circ(\varphi_i^{-1}\times id)\right)\circ\left((\varphi_i\times id)\circ\Phi_i\circ s_m\circ\varphi_i^{-1}\right)\\
&-\left(\psi_i\circ f^*\textup{exp}\circ\Phi_i^{-1}\circ(\varphi_i^{-1}\times id)\right)\circ\left((\varphi_i\times id)\circ\Phi_i\circ s\circ\varphi_i^{-1}\right)\|_{C^k(\overline{\varphi_i(U_i)},\mathbb{R}^n)}\\
&\le C_i\|(\varphi_i\times id)\circ\Phi_i\circ s_m\circ\varphi_i^{-1}-(\varphi_i\times id)\circ\Phi_i\circ s\circ\varphi_i^{-1}\|_{C^k(\overline{\varphi_i(U_i)},\mathbb{R}^n\times\mathbb{R}^n)}\\
&=C_i\|pr_2\circ\Phi_i\circ s_m\circ\varphi_i^{-1}-pr_2\circ\Phi_i\circ s\circ\varphi_i^{-1}\|_{C^k(\overline{\varphi_i(U_i)},\mathbb{R}^n)}\\
&<r,
\end{align*}
where $(f^*\textup{exp})(v):=\textup{exp}_{f(p)}v$ for $v\in(f^*TN)_p$, $p\in M$. We have shown \[H(s_m)\in\bigcap_{i=1}^l\mathcal{N}^k(H(s),\varphi_i,\tilde{U}_i,\psi_i,B_\delta(V_i),\overline{U_i},r)\]
for $m$ large enough, so $H\colon U\rightarrow C^k(M,N)$ is continuous.
\end{proof}

\section{The smooth structure on $C^k(M,N)$}
 In the following we ``globalize'' the local $\Omega$-lemma (Lemma \ref{lemma loc omega}) to sections of vector bundles. This will be the main input for showing that $C^k(M,N)$ carries a \textit{smooth} structure. 
 
 We start with a proposition that provides a criterion for a map with target $\Gamma_{C^k}(E)$ to be a $C^r$-map.

\begin{proposition}\label{proposition 1} In the situation of Definition \ref{def C^k norm vrb}, we define
	\[R_i\colon\Gamma_{C^k}(E)\rightarrow C^k(\overline{\varphi_i(U_i)},\mathbb{R}^n)\]
by $R_i(s):=pr_2\circ\Phi_i\circ s\circ\varphi_i^{-1}$ for $i=1,\ldots,l$, where we assume that $\textup{rank}(E)=n$. Let $r\in\mathbb{N}$, $X$ a Banach space, $U\subset X$ open, and $$F\colon U\rightarrow \Gamma_{C^k}(E)$$ a map. Then $F\in C^r(U,\Gamma_{C^k}(E))$ if and only if $R_i\circ F\in C^r (U,C^k(\overline{\varphi_i(U_i)},\mathbb{R}^n))$ for $i=1,\ldots,l$.
\end{proposition}
\begin{proof}[Sketch of proof]``$\Rightarrow$:'' The $R_i$ are linear and continuous, so the are smooth.\\
	``$\Leftarrow$:'' To make things easier, we first get rid of the $\Phi_i$ and $\varphi_i$ in $R_i\circ F$ as follows: On the vector space
	\[\Gamma_{C^k,\overline{U}_i}(E):=\{s\colon U_i\rightarrow E\text{ }|\text{ } s\in\Gamma_{C^k}(E|_{U_i})\text{ and } pr_2\circ\Phi_i\circ s\circ \varphi_i^{-1}\in C^k(\overline{\varphi_i(U_i)},\mathbb{R}^n)\}\]
	we define the norm
	\[\|s\|_i:=\|pr_2\circ\Phi_i\circ s\circ\varphi_i^{-1}\|_{C^k(\overline{\varphi_i(U_i)},\mathbb{R}^n)}.\]
	We get isomorphisms of Banach spaces
	\begin{align*}
	J_i\colon \Gamma_{C^k,\overline{U}_i}(E)&\rightarrow C^k(\overline{\varphi_i(U_i)},\mathbb{R}^n)\\
	s&\mapsto pr_2\circ\Phi_i\circ s\circ\varphi_i^{-1},
	\end{align*}
	with $J_i^{-1}(f)=\Phi_i^{-1}(id_{U_i},f\circ\varphi_i)$. By assumption, we have that 
	\begin{align*}
	F_i:=J_i^{-1}\circ R_i\circ F\colon U&\rightarrow \Gamma_{C^k,\overline{U}_i}(E),\\
	x&\mapsto F(x)|_{U_i}
	\end{align*}
	is an element of $C^r(U,\Gamma_{C^k,\overline{U}_i}(E))$ for $i=1,\ldots,l$. Define
	\[\tilde{D}^jF\colon U\rightarrow L^j_s(X,\Gamma_{C^k}(E))\]
	by
	\[\left(\tilde{D}^jF\right)_u(x_1,\ldots,x_j)|_{U_i}:=\left(D^jF_i\right)_u(x_1,\ldots,x_j)\]
	for $u\in U$, $x_1,\ldots,x_j\in X$, $j=1,\ldots,r$.
	
	Inductively, one can show that $\tilde{D}^jF$ is well-defined, continuous, and $F$ is $r$ times continuously differentiable with $D^jF=\tilde{D}^jF$ for $j=1,\ldots,r$. Details can be found in \cite[Proof of Proposition 3.4.1.]{JWDissertation}.
\end{proof}

\begin{lemma}[global $\Omega$-lemma] \label{lemma glob omega}Let $r,k\in\mathbb{N}$. Let $M$ be a closed manifold of dimension $m$. Let $E\to M$ be a $C^k$ vector bundle of rank $n$, and let $h$ be a bundle metric on $E$. Choose $U_i,\tilde{U}_i,\hat{U}_i,\varphi_i,\Phi_i$, $i=1,\ldots,l$ as in Definition \ref{def C^k norm vrb} and s.t. the $\Phi_i$ are isometries on the fibers. Let $\delta>0$ and define the open subset $U\subset E$ by
\[U:=\{v\in E\text{ }|\text{ }\|v\|_h<\delta\}.\]
Let $F\to M$ be a $C^k$ vector bundle of rank $d$ with local trivializations $(\hat{U}_i,\tilde{\Phi}_i)$, $i=1,\ldots,l$, and 
\[f\colon U\rightarrow F\]
a map s.t.
\begin{enumerate}
\item $f$ is fiber-preserving and
\item the maps
\[g_i\colon \varphi_i(U_i)\times B_\delta(0)\rightarrow \mathbb{R}^d\]
defined by $$g_i(x,v):=\left(pr_2\circ\tilde{\Phi}_i\circ f\circ\Phi_i^{-1}\circ (\varphi_i^{-1},id)\right)(x,v)$$
for $i=1,\ldots,l$ and $B_\delta(0)\subset\mathbb{R}^n$ the open ball in $\mathbb{R}^n$ of radius $\delta$ and center $0$, satisfy  $$g_i\in C^k(\overline{\varphi_i(U_i)\times B_\delta(0)},\mathbb{R}^d)$$ and for each $j=0,\ldots,r$ the map
\[D^j_2g_i\colon\varphi_i(U_i)\times B_\delta(0)\rightarrow L^j_s(\mathbb{R}^m,\mathbb{R}^d)\]
defined by $(D^j_2g_i)(x,y):=(D^j(g_i(x,.)))(y)$ for all $(x,y)\in\varphi_i(U_i)\times B_\delta(0)$ exists and is an element of $C^k(\overline{\varphi_i(U_i)\times B_\delta(0)},L^j_s(\mathbb{R}^m,\mathbb{R}^d))$.
\end{enumerate}
Then the map
\begin{align*}
\Omega_f\colon \Gamma_{C^k}(E)^U&\rightarrow \Gamma_{C^k}(F),\\
s&\mapsto f\circ s,
\end{align*}
is an element of $C^r(\Gamma_{C^k}(E)^U,\Gamma_{C^k}(F))$ where $\Gamma_{C^k}(E)^U\subset\Gamma_{C^k}(E)$ is the open subset of $C^k$-sections of $E$ with image contained in $U$. If $r\ge 1$, then
\begin{align}\label{eq2}
\left(\left(D\Omega_f\right)_{s_0}s\right)(p)=(D(f|_{E_p\cap U}))_{s_0(p)}s(p)
\end{align}
for all $p\in M$, $s_0\in\Gamma_{C^k}(E)^U$, and all $s\in\Gamma_{C^k}(E)$.
\end{lemma}

A different version of the global $\Omega$-lemma can be found in \cite[Theorem 5.9]{Gl}. (Note that in \cite[Theorem 5.9]{Gl} it is a requirement that the considered map $f$ maps the zero element of each fiber onto itself, $f(0_x)=0_x$. This makes it problematic to apply \cite[Theorem 5.9]{Gl} in our setting, since we will consider maps of the form $v\mapsto \textup{exp}_{g(p)}^{-1}\circ \textup{exp}_{f(p)}v$.)

\begin{remark} \
\begin{enumerate}
		\item Note that in the situation of Lemma \ref{lemma glob omega} ii), the statement
		\[g_i\in C^k(\overline{\varphi_i(U_i)\times B_\delta(0)},\mathbb{R}^d)\text{ and } D^j_2g_i\in C^k(\overline{\varphi_i(U_i)\times B_\delta(0)},L^j_s(\mathbb{R}^m,\mathbb{R}^d))\hspace{3em}\]
		for $j=0,\ldots,r$
		is equivalent to the statement that
		\[\partial^\alpha_y\partial^\beta_xg_i\colon\varphi_i(U_i)\times B_\delta(0)\rightarrow\mathbb{R}^d\]
		are continuous and continuously extendable to $\overline{\varphi_i(U_i)\times B_\delta(0)}$ for all $|\alpha|\le k+r$, $|\beta|\le k$, s.t. $|\alpha+\beta|\le k+r$, where $x$ denotes the ``$\varphi_i(U_i)$-direction'' and $y$ denotes the ``$B_\delta(0)$-direction''.
		\item The assumptions of Lemma \ref{lemma glob omega} ii) imply in particular that $\Omega_f$ is well-defined as a map: from ii) we see that $pr_2\circ\tilde{\Phi}_i\circ f\colon U\cap E|_{U_i}\rightarrow\mathbb{R}^d$ is $C^k$. It follows that $f(v)=\tilde{\Phi}_i^{-1}\circ (\pi,pr_2\circ\tilde{\Phi}_i)(v)$ for all $v\in U\cap E|_{U_i}$,  where $\pi\colon E\rightarrow M$ is the projection of $E$, so $f\in C^k(U\cap E|_{U_i},F)$. Since the $U_i$ cover $M$, we have $f\in C^k(U,F)$ and thus $f\circ s\in \Gamma_{C^k}(F)$ for all $s\in \Gamma_{C^k}(E)^U$.
\end{enumerate}
\end{remark}

\begin{proof}[Proof of Lemma \ref{lemma glob omega}]
For each $i=1,\ldots,l$ we have a commutative diagram
\[
\begin{xy}
  \xymatrix{
      \Gamma_{C^k}(E)^U \ar[r]^{\Omega_f} \ar[dd]_{R_i}    &   \Gamma_{C^k}(F)  \ar[rr]^{\tilde{R}_i} && C^k(\overline{\varphi_i(U_i)},\mathbb{R}^d)  \\ \\
      C^k(\overline{\varphi_i(U_i)},B_\delta(0)) \ar[rrruu]_{\Omega_{g_i}}   
  }
\end{xy}
\]
where $R_i(s):=pr_2\circ\Phi_i\circ s\circ\varphi_i^{-1}$, $\tilde{R}_i(s):=pr_2\circ\tilde{\Phi}_i\circ s\circ\varphi_i^{-1}$, and $\Omega_{g_i}(h)=g\circ (id\times h)$. From Proposition \ref{proposition 1} we see that $\Omega_f$ is $C^r$ iff $\tilde{R}_i\circ\Omega_f$ is $C^r$. Moreover,  $\tilde{R}_i\circ\Omega_f=\Omega_{g_i}\circ R_i$ is $C^r$ because of Lemma \ref{lemma loc omega}, thus $\Omega_f$ is $C^r$.

Equation \eqref{eq2} can be shown by differentiating the above commutative diagram and using equation \eqref{eq4} from Lemma \ref{lemma loc omega}.
\end{proof}

\begin{theorem}[$C^k(M,N)$ as a Banach manifold]\label{theorem ckmn banach mf} Let $k\in\mathbb{N}$. Let $M$ and $N$ be manifolds without boundary. Let $M$ be compact and let $N$ be connected. Choose a Riemannian metric $h$ on $N$. Then the topological space $C^k(M,N)$ (i.e., the set $C^k(M,N)$ equipped with the compact-open $C^k$ topology) has the structure of a smooth Banach manifold such that the following holds: for any $f\in C^k(M,N)$ and any small enough open neighborhood $U_f$ of $f$ in $C^k(M,N)$ there is an open neighborhood $V_f$ of the zero section in $\Gamma_{C^k}(f^*TN)$ such that the map
	\begin{align*}
	\varphi_f\colon U_f&\to V_f,\\
	g&\mapsto \textup{exp}^{-1}\circ (f,g),
	\end{align*}
	i.e., $\varphi_f(g)(p)=(\textup{exp}_{f(p)})^{-1}g(p)$ for all $g\in U_f$, $p\in M$, is a local chart (in particular a smooth diffeomorphism). Note that the inverse of $\varphi_f$ is given by $$\varphi_f^{-1}(s)(p)=\textup{exp}_{f(p)}s(p)$$ for all $s\in V_f$, $p\in M$. This smooth structure does not depend on the choice of the Riemannian metric $h$ on $N$. Moreover, for all $f,g\in C^k(M,N)$ s.t. $U_{f}\cap U_{g}\neq\varnothing$ it holds that
\begin{align}\label{eq 111}
	\left(D(\varphi_g\circ\varphi_f^{-1})_{s_0}s\right)(p)=D(\textup{exp}_{g(p)}^{-1}\circ \textup{exp}_{f(p)})_{s_0(p)}s(p)
\end{align}
for all $p\in M$, $s_0\in\varphi_f(U_{f}\cap U_{g})$, $s\in\Gamma_{C^k}(f^*TN)$.
\end{theorem}

\begin{proof}For $f\in C^k(M,N)$ we denote by $U_{f,\varepsilon}$ the set defined in Lemma \ref{lemma charts c^kMN are homeo}. First we show that for $U_{f,\varepsilon^f}\cap U_{g,\varepsilon^g}\neq \varnothing$ the transition map $\varphi_g\circ\varphi_f^{-1}$ is smooth. We use a strategy similar to the proofs of Lemma \ref{lemma charts c^kMN are homeo} i)-ii). To be more precise, we first show the statement holds for sets $U_{f,\varepsilon^f}$ with some additional assumptions on the charts in the definition of $U_{f,\varepsilon^f}$. We will call these sets $U_{f,\varepsilon^f}^\text{add.}$. Then we show that we don't need these additional assumptions, provided that $\varepsilon^f>0$ and $\varepsilon^g>0$ are small enough. We start by defining the sets $U_{f,\varepsilon^f}^{\text{add.}}$, that is, we formulate which additional assumptions we make on the charts in the definition of $U_{f,\varepsilon^f}$.

Let $f\in C^k(M,N)$. Choose charts $(U_i^f,\varphi_i^f)$ of $M$, $i=1,\ldots,l=l(f)$, $\bigcup_{i=1}^lU_i^f=M$, s.t. $\overline{U_i^f}\subset M$ is compact, $\overline{U_i^f}\subset\tilde{U}_i^f$, $(\tilde{U}_i^f,\varphi_i^f)$ is still a chart of $M$, $f(\overline{U_i^f})\subset V_i^f$, $(V_i^f,\psi_i)$ chart of $N$, $\overline{V_i^f}\subset N$ is compact, $\overline{V_i^f}\subset\tilde{V}_i^f$, $\overline{\tilde{V}_i^f}\subset N$ is compact, and $(\tilde{V}_i^f,\hat{\Phi}_i^f)$ is a local trivialization of $TN$ which is an isometry on fibers for $i=1,\ldots,l$. Choose $$0<r^f<\min_{i=1,\ldots,l(f)}\inf_{q\in \overline{V_i^f}}\textup{inj}_q(N)$$ s.t.
$E$ is a diffeomorphism from the set
\[X_i^f:=\{v\in TN\text{ }|\text{ } \pi_{TN}(v)\in V_i^f, \text{ }\|v\|_h<r^f\}\]
onto its image. Denote by $(f^{-1}(\tilde{V}_i^f),\Phi_i^f)$ the local trivialization of $f^*TN$ induced by $(\tilde{V}_i^f,\hat{\Phi}_i^f)$.
 Now we define the set
\[U_{f,\varepsilon^f}^{\text{add.}}:=\bigcap_{i=1}^l\mathcal{N}^k(f,\varphi_i^f,\tilde{U}_i^f,\psi_i^f,V_i^f,\overline{U_i^f},\varepsilon^f)\]
where $\varepsilon^f =\varepsilon^f(\delta^f)>0$ is chosen s.t. Lemma \ref{lemma charts c^kMN are homeo} i)-iii) hold (for $U_{f,\varepsilon^f}^{\text{add.}}$) where $$\delta^f<\frac{r^f}{6}.$$

Assume that $U_{f,\varepsilon^f}^{\text{add.}}\cap U_{g,\varepsilon^g}^{\text{add.}}\neq\varnothing$. Define 
	\[U:=\{v\in f^*TN\text{ }|\text{ } \|v\|_{f^*h}<2\delta^f\}\]
and
\[F\colon U\rightarrow g^*TN\]
by 
\[F(v):=\left((\textup{exp}_{g(p)})^{-1}\circ\textup{exp}_{f(p)}\right)(v)\] 
for $v\in U\cap T_{f(p)}N$.

After possibly interchanging the roles of $f$ and $g$, we may assume
\begin{align}\label{eq 11111}
\delta^f\le \delta^g.
\end{align}
(It is important to note here, that \eqref{eq 11111} is achieved by possibly interchanging $f$ and $g$. It is \textit{not} achieved by choosing $\delta^f$ so small, that \eqref{eq 11111} holds. The latter would mean that $\delta^f$ also depends on $g$ and then some of the arguments below no longer work.) 
Then
\begin{align} \label{eqqqq 3}
	F(v)=E|_{X_j^g}^{-1}(g(p),\textup{exp}_{f(p)}v)
\end{align}
for all $v\in U\cap T_{f(p)}N$, where $p\in U_j^g$. Hence, $F$ is well-defined. To show \eqref{eqqqq 3}, let $v\in U\cap T_{f(p)}N$ and $p\in U_j^g$. Then $\textup{exp}_{f(p)}v\in B_{2\delta^f}(f(p))$. Since $U_{f,\varepsilon^f}^{\text{add.}}\cap U_{g,\varepsilon^g}^{\text{add.}}\neq\varnothing$, we have $d(f(p),g(p))<\delta^f+\delta^g$ by the triangle inequality. Therefore, $\textup{exp}_{f(p)}v\in B_{3\delta^f+\delta^g}(g(p))$. Since $\delta^f\le \delta^g$ we have $3\delta^f+\delta^g\le 4\delta^g<\frac46 r^g<\frac46\textup{inj}_{g(p)}(N)$. From this it is easy to see that \eqref{eqqqq 3} holds.

Now we want to use Lemma \ref{lemma glob omega} to show that
\begin{align*}
\Omega_F\colon \Gamma_{C^k}(f^*TN)^U&\rightarrow \Gamma_{C^k}(g^*TN),\\
s&\mapsto F\circ s,
\end{align*}
is (well-defined and) smooth. If we have shown that, then we have in particular that $\varphi_g\circ\varphi_f^{-1}=\Omega_F|_{\varphi_f(U_{f,\varepsilon^f}^{\text{add.}}\cap U_{g,\varepsilon^g}^{\text{add.}})}$ is smooth. Condition i) of Lemma \ref{lemma glob omega} is satisfied, as $F$ is fiber-preserving. Now we show Condition ii): To that end, we consider the maps
\[g_{ij}\colon pr_2\circ \Phi_j^g\circ F\circ (\Phi_i^f)^{-1}\circ\left((\varphi_i^f)^{-1},id\right)\colon \varphi_i^f(U_i^f\cap U_j^g)\times B_{2\delta^f}(0)\rightarrow \mathbb{R}^n,\]
where $n=\textup{dim}(N)$ and the maps
\[H_{ij}\colon Y_{ij}\rightarrow TN\]
where $Y_{ij}$ is the non-empty open set
\[Y_{ij}:=\{(q_1,q_2,y)\in V_i^f\times V_j^g\times B_{2\delta^f}(0)\text{ }|\text{ } q_2\in B_{2\delta^f+\delta^g}(q_1)\}\]
and
\[H_{ij}(q_1,q_2,y):=\left((\textup{exp}_{q_2})^{-1}\circ \textup{exp}_{q_1}\right)\left((\hat{\Phi}_i^f)^{-1}(q_1,y)\right).\]
Note that $H_{ij}$ is well-defined and smooth (on $Y_{ij}$) since under our assumption \eqref{eq 11111} it holds that
\[H_{ij}(q_1,q_2,y)=E|_{X_j^g}^{-1}(q_2,\textup{exp}_{q_1}((\hat{\Phi}_i^f)^{-1}(q_1,y)))\]
on $Y_{ij}$.

Moreover, we have
\begin{align}\label{eq6}
pr_2\circ\Phi_j^g\circ H_{ij}\circ\left((f,g)\circ(\varphi_i^f)^{-1},id\right)=g_{ij}
\end{align}
on $\varphi_i^f(U_i^f\cap U_j^g)\times B_{2\delta^f}(0)$. Given any multiindex $\alpha$ we see from \eqref{eq6} that
\[\partial^\alpha_yg_{ij}(x,y)=(pr_2\circ\Phi_j^g)\left(\left(\partial^\alpha_yH_{ij}\right)\left(f((\varphi_i^f)^{-1}(x)),g((\varphi_i^f)^{-1}(x)),y)\right)\right)\]
for all $(x,y)\in \varphi_i^f(U_i^f\cap U_j^g)\times B_{2\delta^f}(0)$, so $\partial^\alpha_yg_{ij}$ is $C^k$ in $(x,y)$. In particular, for $|\beta|\le k$, we have that $\partial^\beta_x\partial^\alpha_y g_{ij}$ is continuous on $\overline{\varphi_i^f(U_i^f\cap U_j^g)\times B_{\delta^f}(0)}$. We have shown Conditions i) and ii) of Lemma \ref{lemma glob omega}, which we now apply to deduce that $\Omega_F\colon \Gamma_{C^k}(f^*TN)^U\rightarrow \Gamma_{C^k}(g^*TN)$ is smooth. In particular, $\varphi_g\circ\varphi_f^{-1}=\Omega_F|_{\varphi_f(U_{f,\varepsilon^f}^{\text{add.}}\cap U_{g,\varepsilon^g}^{\text{add.}})}$ is smooth.

Next we show that we don't need the additional assumptions we made in the definition of the sets $U_{f,\varepsilon^f}^{\text{add.}}$. For arbitrary $U_{f,\varepsilon}$ and $U_{g,\tilde{\varepsilon}}$ (defined as in Lemma \ref{lemma charts c^kMN are homeo}) there exist $U_{f,\varepsilon^f}^{\text{add.}}$ and $U_{g,\tilde{\varepsilon}^g}^{\text{add.}}$ with $U_{f,\varepsilon}\subset U_{f,\varepsilon^f}^{\text{add.}}$ and $U_{g,\tilde{\varepsilon}}\subset U_{g,\tilde{\varepsilon}^g}^{\text{add.}}$, provided that $\varepsilon>0$ and $\tilde{\varepsilon}>0$ are small enough (see Lemma \ref{lemma weak top properties}). If $U_{f,\varepsilon}\cap U_{g,\tilde{\varepsilon}}\neq \varnothing$, then we have in particular $U_{f,\varepsilon^f}^{\text{add.}}\cap U_{g,\tilde{\varepsilon^g}}^{\text{add.}}\neq \varnothing$. We have shown that the transition map $\varphi_g\circ\varphi_f^{-1}$ is smooth on $\varphi_f(U_{f,\varepsilon^f}^{\text{add.}}\cap U_{g,\tilde{\varepsilon^g}}^{\text{add.}})$, so it is in particular smooth on $\varphi_f(U_{f,\varepsilon}\cap U_{g,\tilde{\varepsilon}})$.

Equation \eqref{eq 111} is a direct consequence of equation \eqref{eq2}.

Summing up, we have shown that for $f,g\in C^k(M,N)$ with $U_{f,\varepsilon}\cap U_{g,\tilde{\varepsilon}}\neq\varnothing$ ($\varepsilon, \tilde{\varepsilon}>0$ small enough), the transition map
\[\varphi_g\circ\varphi_f^{-1}\colon \varphi_f(U_{f,\varepsilon}\cap U_{g,\tilde{\varepsilon}})\to \varphi_g(U_{f,\varepsilon}\cap U_{g,\tilde{\varepsilon}})\]
is smooth \textit{after possibly interchanging the roles of $f$ and $g$}, c.f. \eqref{eq 11111}. Hence, we did not show yet that $\varphi_g\circ\varphi_f^{-1}$ is a diffemorphism. Equation \eqref{eq 111} yields that the differential of $\varphi_g\circ\varphi_f^{-1}$ at an arbitrary $s_0\in \varphi_f(U_{f,\varepsilon}\cap U_{g,\tilde{\varepsilon}})$ is a bijective, continuous linear map (between Banach spaces), hence it is a linear isomorphism. The Inverse Mapping Theorem (see e.g. \cite[Theorem 2.5.2]{MTAA}) yields that $\varphi_g\circ\varphi_f^{-1}$ is a diffeomorphism.

A similar argument can be used to show that the above smooth structure does not depend on the choice of the Riemannian metric $h$ on $N$. 
\end{proof}
\begin{remark} For our proofs it was a crucial fact that $M$ is compact. If $M$ is non-compact, the ansatz of using the exponential map of the target manifold to construct the charts still makes sense, but the question of how to topologize $C^k(M,N)$ arises. The case $k=\infty$ was worked out in \cite{KM}.
	
Also, one can consider infinite dimensional target $N$. In the case that $N$ is a Banach manifold admitting partitions of unity, one can work with sprays to construct the charts \cite{Blue}.	
\end{remark}

Lastly, we want to state and prove some mapping properties which can be found in e.g. \cite{Blue}.
\begin{proposition}[Mapping properties] Let $k,r\in\mathbb{N}$. Let $M, N, A, Z$ be manifolds without boundary. Assume that $M$ and $A$ are compact. Moreover, assume that $N$ and $Z$ are connected. Then the following holds:
	\begin{enumerate}
		\item If $g\in C^{k+r}(N,Z)$, then the map
		\begin{align*}
			\omega_g\colon C^k(M,N)&\to C^k(M,Z),\\
			f&\mapsto g\circ f,
		\end{align*}
		is $C^r$.
		\item If $g\in C^k(A,M)$, then the map
		\begin{align*}
		\alpha_g\colon C^k(M,N)&\to C^k(A,N),\\
		f&\mapsto f\circ g,
		\end{align*}
		is smooth.		
	\end{enumerate}
\end{proposition}
\begin{proof} To show assertion i) of the proposition, one uses the local charts from Theorem \ref{theorem ckmn banach mf} to write down the local representative of $\omega_g$ to which Lemma \ref{lemma glob omega} is applied.
	
	 Assertion ii) of the proposition is shown in a similar manner. First, one uses the charts of $C^k(M,N)$ to reduce to the situation of maps between sections of vector bundles. Using local trivializations of these vector bundles, the problem is further reduced to the following statement: Let $Y$ be a Banach space. Let $U\subset \mathbb{R}^n$ and $V\subset\mathbb{R}^m$ be open and bounded. Given $g\in C^k(\overline{V}, U)$, the map
		\begin{align*}
		\tilde{\alpha}_g\colon C^k(\overline{U},Y)&\to C^k(\overline{V},Y),\\
		f&\mapsto f\circ g,
		\end{align*}
		is smooth. This however is clear, since $\tilde{\alpha}_g$ is linear and continuous.	
\end{proof}

\subsection*{Acknowledgments}The author would like to thank Bernd Ammann and Olaf Müller for interesting discussions about this topic. We thank Peter Michor and Alexander Schmeding for providing useful references and comments. The author's work was supported by the DFG Graduiertenkolleg GRK 1692 ``Curvature, Cycles, and Cohomology''.

\newpage

\bibliographystyle{abbrv}
\bibliography{references}

\end{document}